\documentclass[a4paper,reqno]{amsart}
\numberwithin{equation}{section}

\newtheorem{theorem}{Theorem}[section]

\newtheorem{lemma}[theorem]{Lemma}

\newtheorem{corollary}[theorem]{Corollary}

\def\beq{\begin{equation}}
\def\eeq{\end{equation}}
\def\be{\begin{equation*}}
\def\ee{\end{equation*}}

\usepackage{amssymb}
\usepackage{amsfonts}

\title{Interpreting the weak monadic second order theory of the ordered rationals}
\author{J K Truss} 

\begin{document}

\begin{abstract} We show that the weak monadic second order theory of the structure $({\mathbb Q}, <)$ is first order interpretable in
its automorphism group.
\end{abstract}

\maketitle  

\setcounter{footnote}{1}\footnotetext{
keywords: weak monadic second order logic, ordered rationals, interpretation}
\newcounter{number}

\section {Introduction}

The {\em monadic second order theory} of a structure is the set of second order sentences true in the structure, 
for which second order quantification is only performed over subsets of the domain (i.e. unary predicates). Its 
{\em weak monadic second order theory} instead allows quantification just over the {\em finite} subsets of the structure.
Marcus Tressl enquired whether the weak monadic second order theory of $({\mathbb Q}, <)$ can be interpreted inside
its endomorphism monoid. Here we show that this is indeed possible, and in fact it can be interpreted inside its
automorphism group.

We shall use similar methods as described in \cite{Truss}. There the monoids of monomorphisms, endomorphisms 
of $({\mathbb Q}, <)$ were denoted by $M$ and $E$ respectively, and its group of automorphisms by $G$. It was shown
that the action of $E$ on $\mathbb Q$ is interpretable in the monoid $(E, \circ)$. This is done by means of a series
of first order formulae of the language of group theory. We require some of these here, and so shall recap the main 
ideas without full details, for which we refer the reader to \cite{Truss}. In fact for our present purposes, it
suffices to work just with the group $G$. The corresponding results follow easily for $M$ and $E$, since $G$ is a definable 
subset of each of these. 

\section {Background}

Much of the material needed is given, either explicitly or implicitly, in \cite{Glass}. We refer mainly to the
presentation as given in \cite{Truss}. 

The key to unlocking the properties of $G$ is the notion of `orbital' of a member of $G$. This is defined to be the 
convex closure of an orbit, and which carries a `parity', $+1$, $-1$, or 0, depending on whether the map is increasing, 
decreasing, or fixed, on that orbital. More precisely, an {\em orbital} of $f \in G$ is an equivalence class under the 
relation given by $a \sim b$ if for some integers $m, n$, $f^ma \le b \le f^na$. It is easily seen that for any 
orbital $X$ of $f$, for every element $a$ of $X$, $a < fa$, or for every $a$, $a > fa$, or there is just one element
of $X$, which is fixed by $f$. We say that $X$ has {\em parity} $+1$, $-1$ or 0 in these three cases. Orbitals of 
parity $\pm 1$ are called {\em non-trivial}. Since all orbitals are convex, the family of orbitals receives the natural 
induced ordering, and can therefore be viewed as a 3-coloured linear order, referred to as its `orbital pattern'. 
Furthermore, two group elements are conjugate if and only if their orbital patterns are isomorphic (as coloured 
orders) (see for instance \cite{Glass}). We write the conjugate $gfg^{-1}$ of $f$ by $g$ as $f^g$. 
        
We omit the precise details by which the following formulae of the language of group theory are constructed, as these 
are given fully in \cite{Truss}, just sketching the intuition:

${\bf comp}(x)$ is a formula expressing `comparability' with the identity, so that for $f \in G$, ${\bf comp}(f)$ 
holds in $G$ if and only if either for all $a$, $a \le fa$, or for all $a$, $a \ge fa$.

${\bf apart}(x, y)$ expresses that the support of $x$ is either entirely to the left of that of $y$, or entirely to its 
right (including the vacuous case that one of both of these supports is empty), where the {\em support} of a group
element is the set of points moved by it.

${\bf bump}(x)$ expresses that $x$ is a `bump', which is defined to be a non-identity element having exactly one 
non-trivial orbital.

${\bf orbital}(x, y)$ expresses that $x$ is an orbital of $y$, meaning that $x$ is a bump whose support is contained in that
of $y$, and such that $x$ is equal to the restriction of $y$ to its support. 

${\bf disj}(x,y)$ is a formula of the language of group theory such that for $f, g \in G$, ${\bf disj}(f,g)$ holds in
$G$ if and only if $f$ and $g$ have disjoint supports. From this we can derive a formula ${\bf restr}(x,y)$ which says 
that the support of $x$ is contained in that of $y$, and the restrictions of $x$ and $y$ to the support of $x$ are equal. 
This formula ${\bf restr}(x,y)$ is just $\exists z({\bf disj}(x,z) \wedge y = xz)$.  The formula ${\bf cont}(x,y)$
says that the support of $x$ is contained in that of $y$, and this is taken to be 
$\forall z({\bf disj}(y,z) \to {\bf disj}(x,z))$.

There are two particular types of bump which will be needed, which can each be characterized by a formula. A bump $f$ is 
said to be {\em coterminal} if its support is the whole of $\mathbb Q$, which may be expressed by saying that it is a 
bump which is not disjoint from any non-identity member of $G$, written ${\bf coterm}(f)$. It is said to be {\em cofinal}
if its support is bounded above or below but not both. Saying that $f$ is cofinal can be expressed by a formula {\bf cof} 
expressing that it is not coterminal, and it is not disjoint from any conjugate. Cofinal elements are particularly 
important, since they will have support $(-\infty, a)$ or $(a, \infty)$ for some $a \in {\mathbb R}$, so can be used to 
encode the endpoint $a$ (which may be rational or irrational, though we really want just the rational case, and have to
show in addition how this can be expressed). 

\begin{lemma}  \label{2.1}  If $f \in G$ has infinitely many non-trivial orbitals, then it has a non-trivial restriction
$g$ which may be written as $g_1g_2$ where $g_1$ is an orbital of $g$, and $g$ is conjugate to $g_2$. 
\end{lemma}

\begin{proof} Since $f$ has infinitely many non-trivial orbitals, it has either an increasing or decreasing $\omega$-sequence 
of such non-trivial orbitals. Without loss of generality assume this is increasing,  $X_0 < X_1 < X_2 < \ldots$ say. 
By passing to a suitable subsequence, we may assume that all these orbitals have the same parity ($+1$ or $-1$), $X_0$ 
is bounded below (i.e. does not have $-\infty$ as its left endpoint), and that for each $n$, ${\rm sup} \, X_n < {\rm inf} \, X_{n+1}$. 
Furthermore, if we write $X_n = (a_n, b_n)$, we may suppose that either all $a_n$ are rational, or all are irrational, and 
similarly for the $b_n$ (since there are only 4 possibilities, this can be achieved by `thinning out'). Let $g_1$ be the
restriction of $g$ to $X_0$, and $g_2$ be the restriction of $g$ to $\bigcup_{n > 0}X_n$. Then $g$ and $g_2$ have 
isomorphic orbital patterns, and hence are conjugate, so $g_1$ and $g_2$ are as desired.    \end{proof}

The lemma leads us to consider the formula ${\bf inf}(x)$ which says that $x$ has a restriction $y$ and there is a non-trivial 
orbital $y_1$ of $y$ such that if $y = y_1y_2$, then $y$ is conjugate to $y_2$.

\begin{lemma}  \label{2.2}  For any $f \in G$, $G \models {\bf inf}(f)$ if and only if $f$ has infinitely many non-trivial 
orbitals. \end{lemma}

\begin{proof} Let $G \models {\bf inf}(f)$, and write $g = g_1g_2$ for a restriction of $f$ with a non-trivial orbital
$g_1$ as provided by the formula. Thus $g$ is conjugate to $g_2$, and as $g_2$ has one fewer non-trivial orbital than
$g$, there must be infinitely many (for each). Hence $f$ also has infinitely many non-trivial orbitals. 

Conversely, by Lemma \ref{2.1}, if $f$ has infinitely many non-trivial orbitals, the formula {\bf inf} must be true
for $f$ in $G$.   \end{proof}

\begin{corollary}  \label{2.3}  For any $f \in G$, $G \models \neg {\bf inf}(f)$ if and only if $f$ has only finitely many 
non-trivial orbitals. \end{corollary}

Having characterized finiteness in one setting, that is, for the number of non-trivial orbitals of a member of $G$, we have 
to transfer it to the interpretation of $\mathbb Q$ found in \cite{Truss}. We recall in outline how this is carried out. The 
method is to use cofinal members of $G$ having a single non-trivial orbital of the form $(-\infty, q)$ or $(q, \infty)$ 
for some rational number $q$, and then to show that there is a formula which identifies two such if and only if 
they correspond to the same value of $q$. To perform the first task, we start by identifying cofinal members of $G$ having a 
single orbital of the form $(-\infty, a)$ or $(a, \infty)$ for the same real number $a$ by means of a formula ${\bf codesame}$. 
For this we note that such elements $f$ and $g$ will either have the same support, which is expressed by the formula 
${\bf cont}(f, g) \wedge {\bf cont}(g, f)$, or `opposite' supports (i.e. one $(-\infty, a)$ and the other $(a, \infty)$), 
which is expressed by a formula ${\bf oppsupport}(f, g)$, which says that they are cofinal bumps which are disjoint, and 
such that no non-identity member of $G$ is disjoint from both of them.

The harder task is to characterize which such elements correspond to $a \in {\mathbb Q}$. (One observes that there are 8
conjugacy classes of cofinal elements $f$, corresponding to $f$ having parity $\pm 1$, support bounded above or below, and
to $a$ rational or irrational.) 

The first task is performed using a formula `{\bf rational}'. Details are given in \cite{Truss}, but we recall the ideas 
in outline here. This is where coterminal elements are required. A typical coterminal element is translation by 1 to the 
right. And, actually, any coterminal element is {\em conjugate} to the element of this form (or its inverse), so in a sense, 
all such elements are (possibly `distorted') translations. Related to this is a formula ${\bf gauge}(x,y)$ which says 
that $x$ and $y$ are commuting coterminal elements whose joint centralizer is commutative. To see that such elements exist, 
consider $\mathbb Q$ replaced by ${\mathbb Q}[\sqrt{2}]$, which being countable dense without endpoints is order-isomorphic 
to $\mathbb Q$, and let $f$ and $g$ be translations by 1 and $\sqrt{2}$ respectively. One establishes, by extending 
to $\mathbb R$ and using a density and continuity argument that ${\bf gauge}(f,g)$ holds (the key point being that the
set of reals of the form $a + b\sqrt{2}$ for $a, b \in {\mathbb Z}$ is dense in $\mathbb R$). It can be shown that this 
situation is essentially typical, that is if ${\bf gauge}(f,g)$ holds then for some irrational $\alpha$, $\mathbb Q$
can be replaced by ${\mathbb Q}[\alpha]$ in such a way that for all $a$, $f(a) = a + 1$ and $g(a) = a + \alpha$.

The formula ${\bf gauge}$ is now used to help us characterize cofinal elements having support $(-\infty, q)$ or $(q, \infty)$ 
for some rational $q$. The main point is that if ${\bf gauge}(f,g)$, then the joint centralizer of $f$ and $g$ is a 
countable group. Up to equivalence under ${\bf codesame}$, there are two orbits of cofinal elements, corresponding
to $q$ rational and $q$ irrational. We use ${\bf gauge}$ to enable us to tell these apart. More precisely, the formula
${\bf rational}(x)$ is built up as follows. It says that $x$ is cofinal, and there are $y$ and $z$ such that 
${\bf gauge}(y,z)$ and for any conjugate $t$ of $x$, there is a conjugacy $u$ of $x$ to a cofinal element $x^u$ 
such that ${\bf codesame}(t, x^u)$, and such that $u$ commutes with both $y$ and $z$. Since as just remarked, such $u$
can take only countably many possible values, the conjugates of $x$ can only encode countably many points, from 
which it follows that $x$ has support $(-\infty, q)$ or $(q, \infty)$ for some rational $q$.

\section{The main result}

We can now put together what we have succeeded in doing, and obtain our main results.

\begin{theorem}  \label{3.1}  The weak monadic second order theory of the structure $({\mathbb Q}, <)$ is first order
interpretable in its automorphism group $G$. More precisely, there are formulae `{\bf finrational}' and 
`{\bf sameset}' of the language of group theory, such that for any $f \in G$, $G \models {\bf finrational}(f)$ if and
only if $f$ is either positive or negative, having only finitely many orbitals, all of whose (finite) endpoints are 
rational; and for any $f, g \in G$, $G \models {\bf sameset}(f,g)$ if and only if 
$G \models {\bf finrational}(f) \wedge {\bf finrational}(g)$ and $f$ and $g$ have the same set of fixed points.
  \end{theorem}
  
\begin{proof}  Note that here, when we say `finitely many orbitals', we mean including trivial ones. In \cite{Truss} 
Theorem 1.13 it was shown that $\mathbb Q$ may be represented inside $G$ by means of elements satisfying the formula 
{\bf rational}, two of which are identified if they satisfy {\bf codesame}. We now have all the ingredients to extend 
this to interpret also finite sets of rationals. For this, we let ${\bf finrational}(x)$ be the formula 

$${\bf comp}(x) \wedge \neg {\bf inf}(x) \wedge \forall y({\bf disj}(x,y) \to y = 1) \wedge$$ 
$$(\forall y, z){\bf oppsupport}(y, z) \wedge {\bf cont}(x, yz) \to {\bf rational}(y).$$

\vspace{.1in}

Deciphering the clauses here, they say

\vspace{.05in}

${\bf comp}(x)$: $x$ is either positive or negative,

\vspace{.05in}

$\neg {\bf inf}(x)$: $x$ has finitely many non-trivial orbitals,

\vspace{.05in}

$\forall y({\bf disj}(x,y) \to y = 1)$: $x$ has dense support (and so by the previous line has only finitely many fixed points),

\vspace{.05in}

$(\forall y, z) {\bf oppsupport}(y, z) \wedge {\bf cont}(x, yz) \to {\bf rational}(y)$: all fixed points of $x$ are rational. 

\vspace{.05in}

For ${\bf sameset}(x, y)$ we use the formula 

$${\bf finrational}(x) \wedge {\bf finrational}(y) \wedge {\bf cont}(x, y) \wedge {\bf cont}(y, x).$$  \end{proof}

Let us make more explicit how these formulae effect the interpretation in $G$ of the weak monadic second order theory of
$({\mathbb Q}, <)$. The idea of the proof just given is that we are using elements having finitely many orbitals of
the form  $(-\infty, a_1)$, $(a_1, a_2), \ldots , (a_n, \infty)$ for some rational numbers $a_1 < a_2 < \ldots < a_n$, 
all of the same parity. (This stands for the finite set $\{a_1, a_2, \ldots, a_n\}$.) It is also (and must be) asserted 
that there is a formula (${\bf sameset}$) telling us when two such elements correspond to the same finite sets of 
rationals. So they do `encode' the set of finite sets of rationals, since clearly every finite set of rationals can arise 
in this way (even the empty set). In addition, relating the interpretations of rationals and finite sets of rationals, we
note that the rational $q$ lies in the finite set $\{a_1, \ldots, a_n\}$ of rationals precisely if $q$ can be 
represented by $f$ such that $G \models {\bf rational}(f)$ and $\{a_1, \ldots, a_n\}$ can be represented by $g$ such that 
$G \models {\bf finrational}(g)$, and for some $f'$ such that $G \models {\bf oppsupport}(f, f')$, we have 
$G \models {\bf cont}(g, ff')$.

To see that this amounts to an interpretation of weak monadic second order logic, we observe that any quantification 
over the set of finite subsets of $\mathbb Q$ can be replaced by quantification over elements satisfying $\bf finrational$.  
  
\begin{theorem}  \label{3.2}  The weak monadic second order theory of the structure $({\mathbb Q}, <)$ is first order
interpretable in each of its monoids $M$ of embeddings and $E$ of endomorphisms.
  \end{theorem}

This follows from the facts that $G$ is a definable subset of each of $M$ and $E$, being its set of invertible elements.
So we can just use the interpretation already given.

\end{document}